\documentclass[11pt,twoside]{article}

\usepackage{mathrsfs,amsfonts,amsmath,amssymb}
\usepackage{latexsym}
\usepackage{comment}

\newtheorem{satz}{Theorem}
\newtheorem{proposition}[satz]{Proposition}
\newtheorem{theorem}[satz]{Theorem}
\newtheorem{lemma}[satz]{Lemma}
\newtheorem{definition}[satz]{Definition}
\newtheorem{corollary}[satz]{Corollary}
\newtheorem{remark}[satz]{Remark}
\newtheorem{example}[satz]{Example}

\newtheorem{problem}[satz]{Problem}

\def\F{\mathbb {F}}
\def\E{\mathsf{E}}

\def\C{\mathbb{C}}

\def\d{\delta}
\def\o{\omega}
\def\({\big (}
\def\){\big )}
\def\g{\gamma}
\def\G{\Gamma}

\def\ls{\leqslant}

\def\dim{{\rm dim}}
\def\codim{{\rm codim}}
\def\le{\leqslant}
\def\ge{\geqslant}
\def\_phi{\varphi}
\def\eps{\varepsilon}

\def\Gr{{\mathbf G}}
\def\FF{\widehat}

\def\ov{\overline}
\def\Spec{{\rm Spec\,}}

\def\t{\tilde}
\def\Span{{\rm Span\,}}

\def\la{\lambda}
\def\D{\Delta}

\def\supp{\mathsf{supp}}

\def\C{\mathbb{C}}
\def\M{\mathsf{M}}

\def\B{{\mathcal B}}
\def\M{{\mathcal M}}

\newcommand{\bp}{\bigskip}

\author{I.D. Shkredov}
\title{
    On Fourier coefficients of sets with small doubling 
}
\date{}
\begin{document}
	\maketitle


\begin{center}
	Annotation.
\end{center}

{\it \small
    Let $A$ be a subset of a finite abelian group such that $A$ has a small difference set $A-A$ and the density of $A$ is small. 
    We prove that, 
    counter--intuitively, 
    the smallness (in terms of $|A-A|$) of the Fourier coefficients of $A$ guarantees that $A$ is correlated with a large Bohr set.  
    Our bounds on the size and the dimension of the resulting Bohr set are close to exact. 
}
\\

\section{Introduction}

Let $\Gr$ be a
finite 
abelian group and  $\FF{\Gr}$  its dual group.
For any function $f:\Gr \to \mathbb{C}$ and $\chi \in \FF{\Gr}$ define the Fourier transform of $f$ at $\chi$ by the formula 
\begin{equation}\label{f:Fourier_representations}
\FF{f} (\chi) = \sum_{g\in \Gr} f(g) \ov{\chi (g)} \,.
\end{equation}
This paper considers the case when the function $f$ has a special form, namely, $f$  the characteristic function of a set $A \subseteq \Gr$ 
and we want to study the quantity 
\begin{equation}\label{def:M(A)}
    \mathcal{M} (A) := \max_{\chi\neq \chi_0} |\FF{A} (\chi)| \,,
\end{equation}
where by $\chi_0$  we have denoted the principal character of $\FF{\Gr}$ and 
we use the same capital letter to denote a set $A\subseteq \Gr$ and   its characteristic function $A: \Gr \to \{0,1 \}$.
It is well--known that $\M(A)$ is directly related to the uniform distribution properties of the set $A$ (see, e.g., \cite{IK_book}, \cite{TV}). 
Moreover, we impose an additional 
property on the set $A$, namely, that it is a set with small {\it doubling}, i.e. we consider sets $A$ for which the ratio $|A-A|/|A|$ is small.  
Recall that given two sets $A,B\subseteq \Gr$ 
the {\it sumset} 
of $A$ and $B$ is defined as 
$$A+B:=\{a+b ~:~ a\in{A},\,b\in{B}\}\,.$$
In a similar way we define the {\it difference sets} $A-A$ and the {\it interated  sumsets}, e.g., $2A-A$ is $A+A-A$.
In terms of the difference sets $|A-A|:=K|A|$ it is easy to see that 
\begin{equation}\label{f:M^2_intr}
    \M^2(A) \ge (1-o(1)) \cdot \frac{|A|^2}{K} 
\end{equation}
and estimate \eqref{f:M^2_intr} gives us a simple lower bound for the quantity $\M(A)$. 
In general, bound \eqref{f:M^2_intr} is tight, but the situation changes dramatically  depending on the {\it density} $\d:=|A|/|\Gr|$ of $A$. 
For example, in \cite[Lemma 4.1]{GR_rectification} (also, see \cite{Schoen_2.4_pol} and \cite[Propostion 6.1]{sanders2012bogolyubov}) it was proved that 
\begin{equation}\label{f:GR_intr}
    \M^2 (A) \ge (1-o(1)) \cdot |A|^2 \,,
\end{equation}
provided $\d = \exp (-\Omega (K))$.  
Thus, if the density is 
exponentially 
small relative to $K$, then $\M(A)$ is close to its maximum value $|A|$.
On the other hand, it is easy to see we always have $\d \le K^{-1}$ and for a random set $A \subseteq \Gr$ one has $\d \gg K^{-1}$. 
Therefore, one of the reasonable  
questions here is the following.
Assume that $\d \sim K^{-d}$, where $d>1$, that is, the dependence of $\d$ on $K$ is polynomial.
What non--trivial properties do the Fourier coefficients of set $A$ have in this case?
Such a problem arises  naturally in connection with the famous Freiman $3k-4$ theorem in $\F_p$, see, e.g., \cite{Freiman_2.4_Fp}, \cite{LS_towards}, \cite{LS_2.6}. 
In particular, in \cite[Theorem 6]{LS_2.6} it was  proved that if $\d \ll K^{-3}$, then
\begin{equation}\label{f:kappa_0_intr}
    \M^2 (A) \ge \frac{|A|^2}{K} (1+\kappa_0) \,,
\end{equation}
where $\kappa_0 >0$ is an absolute constant. 
Thus, a non-trivial lower bound for $\M(A)$ exists in any abelian group, and hence the Fourier coefficients of sets with small doubling have some interesting properties.
Other results on the quantity $\M(A)$ and its relation to sets $A_x$ (see Section \ref{sec:def} below) were obtained in \cite{sh_uncertainty}.  
The main result of this paper (all required definitions can be found in Sections \ref{sec:def}, \ref{sec:general_case}) concerns an even broader regime $\d \ll K^{-2}$ and gives us a new structural property of sets with small doubling and small Fourier coefficients. 


\begin{theorem}
    Let $\Gr$ be a finite abelian group, $A \subseteq \Gr$ be a set, $|A|= \d |\Gr|$, 
    $|A-A|=K|A|$, and $1\le M \le K$ be a parameter. 
    Suppose that 
\begin{equation}\label{cond:M_large_Bohr_intr}
    100 K^2 \d \le 1 \,.
\end{equation}
    Then either there is $x\neq 0$ such that
\begin{equation}\label{f:M_large_Bohr_intr}
    |\FF{A} (x)|^2 > \frac{M|A|^2}{K} \,,
\end{equation}
    or for any $B\subseteq A$, $|B| \gg |A| =  \d |\Gr|$ there exists  a regular Bohr set 
    $\B_*$  
    and  $z\in \Gr$ such that 
\begin{equation}\label{cor:M_intersection_intr}
    |B\cap (\B_* + z)| \ge \frac{|\B_*|}{8M} \,, 
\end{equation}
    and 
\begin{equation}\label{cor:M_Bohr_dim_intr}
    \dim (\B_*) \ll M^2 \left( \log (\d^{-1} K) + \log^2 M \right) \,,
\end{equation}
    as well as 
\begin{equation}\label{cor:M_Bohr_size_intr}
 |\B_*| 
    \gg 
    |\Gr|  \cdot \exp (- O(\dim (\B_*) \log (M \dim (\B_*))) ) 
    \,. 
\end{equation}
\label{t:main_intr}
\end{theorem}

Let us make a few  remarks regarding Theorem \ref{t:main_intr}.
First of all, the bounds \eqref{cor:M_intersection_intr}---\eqref{cor:M_Bohr_size_intr} hold for {\it any} dense subset of our set $A$, which means that $A$ has a very rigid structure
(for example, see the second part of Corollary \ref{cor:M} below). 
Secondly, the dependencies on parameters in \eqref{cor:M_intersection_intr}---\eqref{cor:M_Bohr_size_intr} have polynomial nature, which distinguishes them  from the best  modern results on the structure of sets with small doubling, see \cite{sanders2012bogolyubov}, \cite{sanders2013structure}. 
Also, it appears that Theorem \ref{t:main_intr} does not depend on the recent breakthrough  progress concerning Polynomial Freiman--Ruzsa Conjecture \cite{GGMT_Marton}, \cite{GGMT_Marton_bounded}. 
Thirdly, the inversion of bound  \eqref{f:M_large_Bohr_intr} guarantees the existence of a large intersection of $A$ and a translation of a Bohr set $\B_*$, see inequality \eqref{cor:M_intersection_intr}. 
This is quite surprising, because usually we have the opposite picture: small Fourier coefficients help {\it to avoid} such structural objects as Bohr sets.
Finally,  our proof 
drastically differs  from the arguments of  \cite{GR_rectification}, \cite{Schoen_2.4_pol}, \cite{sanders2012bogolyubov}, which give us  inequality 
\eqref{f:GR_intr}. 
Instead of using almost periodicity of multiple convolutions \cite{CS} and the polynomial growth of sumsets, we apply some observations from the higher energies method, see papers \cite{SS_higher}, \cite{sh_uncertainty}.
In particular, our structural subset of $A-A$ has a different nature (it resembles some steps of the proof of \cite{GGMT_Marton}, \cite{GGMT_Marton_bounded} and even older constructions of Schoen, see \cite[Examples 5,6]{sh_str_survey}). 
Let us consider the following motivating example (see, e.g., \cite[Example 5]{sh_str_survey}).

\begin{example} ($H+\Lambda$ sets).
    Let $\Gr= \F_2^n$, $H\le \Gr$ be the space spanned by the first $k<n$ coordinate vectors, $\Lambda \subseteq H^\perp$ be a basis, $|\Lambda|=2K$, $K \to \infty$, and $A:=H+\Lambda$. 
    In particular, $|A| = |H| |\Lambda|$ and $|A-A| = (1+o(1)) \cdot K|A|$.
    Further 
    for $s\in H$ one has $A_s = A$, but it is easy to check that for $s\in (A-A)\setminus H$ each $A_s$  is a disjoint union of two shifts of $H$. 
    Hence 
\[
    \E(A) \sim |A|^3/K \sim \sum_{s\in H} |A_s|^2 \sim \sum_{s\in (A-A)\setminus H} |A_s|^2  
\]
    (the existence of two ``dual'' subsets of $A-A$ where $\E(A)$ is achieved   is in fact a general result, see \cite{Bateman-Katz_AP3} and also  \cite{sh_str_survey}), 
    but for $k>2$ the higher energies are  supported on the set of measure zero (namely, $H$) in the sense that 
\begin{equation}\label{tmp:E_k_first}
    \E_k (A) := \sum_s |A_s|^k = (1+o_k(1))  \cdot \sum_{s\in H} |A_s|^k \,.
\end{equation}
    Moreover if one considers the function 
$
    \_phi_k (s) = |A_s|^k \,,
$
    then it is easy to see that 
\[
    \FF{\_phi}_k (\chi) := \sum_s |A_s|^k \chi (s) = (1+o_k (1)) \cdot \sum_{s\in H} |A_s|^k \chi (s)  =  (1+o_k(1)) \cdot |H|^k \FF{H} (\chi)  \,,
\]
    where $\chi$ is 
    an arbitrary additive character on $\Gr$. 
    Thus, as $k$ goes to infinity, the Fourier transform $\FF{\_phi}_k (\chi)$ becomes very regular, 
    and this gives us a new way to extract the structure piece from sets  with small doubling. 
\label{exm:H+L_intr}
\end{example}

Roughly speaking, the same Example \ref{exm:H+L_intr} shows that Theorem \ref{t:main_intr} is tight (also, see more rigorous Example \ref{exm:Andersson} below). 
Indeed,  it is easy to see that $\FF{A} = \FF{H} \cdot \FF{\Lambda}$ and assuming that $\Lambda$ is  a sufficiently small set and $\M^2(\Lambda) \ll |\Lambda| \ll K$, we get 
$\M^2(A) = |H|^2 \M^2(\Lambda) \ll |A|^2/K$. 
Thus 
all conditions of Theorem \ref{t:main_intr} 
hold 
but the obvious subspace in $A-A$ is of course $H$ 
and we have 
$\codim (H) = \log (\d^{-1} K)$, which coincides  with \eqref{cor:M_Bohr_dim_intr} up to multiple constants.




\section{Definitions and notation}
\label{sec:def}

Let $\Gr$ be a finite abelian group and we denote the cardinality of $\Gr$ by $N$. 
Given a set $A\subseteq \Gr$ and a positive integer $k$, let us put 
$$
    \Delta_k (A) := \{ (a,a, \dots, a) ~:~ a\in A \} \subseteq \Gr^k \,.
$$
Now we have 
\begin{equation}\label{def:A-A_intersection}
    A-A := \{ a-b ~:~ a,b\in A\} = \{ x\in \Gr ~:~ A\cap (A+x) \neq \emptyset \} \,,
\end{equation}
and for $x\in A-A$ we denote by $A_x$ the intersection $A\cap (A+x)$. 
The useful inclusion of Katz--Koester \cite{Katz-Koester} is the following 
\begin{equation}\label{f:Katz-Koester}
    B+A_x \subseteq (A+B)_x \,.
\end{equation}
A natural generalization of the last formula in \eqref{def:A-A_intersection} 
is the set 
\begin{equation}\label{def:A-A_m_intersection}
    \{ (x_1, \dots, x_k) \in \Gr^k ~:~ A\cap (A+x_1) \cap \dots \cap (A+x_k) \neq \emptyset \} 
    =
    A^k - \Delta_k (A) \,,
\end{equation}
which is called the {\it higher difference set} (see \cite{SS_higher}). 
For any two sets $A,B \subseteq \Gr$ the {\it additive energy} of $A$ and $B$ is defined by
$$
\E (A,B) = \E^{} (A,B) = |\{ (a_1,a_2,b_1,b_2) \in A\times A \times B \times B ~:~ a_1 - b^{}_1 = a_2 - b^{}_2 \}| \,.
$$
If $A=B$, then  we simply write $\E^{} (A)$ for $\E^{} (A,A)$.
Also, one can define  the {\it higher energy} (see \cite{SS_higher} and another formula \eqref{f:Ek_Cf} below)  
\begin{equation}\label{f:E_k2}
      \E_{k}(A) =  | \{ (a_1,\dots,a_k,a'_1,\dots, a'_k) \in A^{2k} ~:~ a_1-a'_1 = \dots = a_k - a'_k \} | \,. 
\end{equation}

\bp 

Let $\FF{\Gr}$ be its dual group.
For any function $f:\Gr \to \mathbb{C}$ and $\chi \in \FF{\Gr}$ we define its Fourier transform using the formula \eqref{f:Fourier_representations}.
The Parseval identity is 
\begin{equation}\label{F_Par}
    N\sum_{g\in \Gr} |f(g)|^2
        =
            \sum_{\chi \in \FF{\Gr}} \big|\widehat{f} (\chi)\big|^2 \,.
\end{equation}
Given a function $f$ and $\eps \in (0,1]$ define the {\it spectrum} of $f$ as 
\begin{equation}\label{def:spectrum}
    \Spec_\eps (f) = \{ \chi \in \FF{\Gr} ~:~ |\FF{f} (\chi)|\ge \eps \| f\|_1 \} \,.
\end{equation}
We need Chang's lemma \cite{chang2002polynomial}.

\begin{lemma}
    Let $f: \Gr \to \C$ be a function  and $\eps \in (0,1]$ be a parameter.
    Then 
\[
    \dim (\Spec_\eps (f)) \ll \eps^{-2} \log \frac{\| f\|_2^2 N}{\| f\|_1^2} \,.
\]
\label{l:Chang}
\end{lemma}

By  {\it measure} we mean any non--negative function $\mu$ on $\Gr$ such that $\sum_{g\in \Gr} \mu (g) = 1$. 
If $f,g : \Gr \to \C$ are some functions, then 
$$
    (f*g) (x) := \sum_{y\in \Gr} f(y) g(x-y) \quad \mbox{ and } \quad (f\circ g) (x) := \sum_{y\in \Gr} f(y) g(y+x) \,.
$$
One has 
\begin{equation}\label{f:F_svertka}
    \FF{f*g} = \FF{f} \cdot \FF{g} \,.
\end{equation}
Having a function $f:\Gr \to \C$ and a positive integer $k>1$, we  write $f^{(k)} (x) = (f\circ f \circ \dots \circ f) (x)$, where the convolution $\circ$ is taken $k-1$ times. 
For example, $A^{(4)} (0) = \E(A)$ and for $k\ge 2$ one has 
\begin{equation}\label{f:Ek_Cf}
    \E_k (A) =   \sum_{x} (A\circ A)^k (x) =
    \sum_{x_1,\dots,x_{k-1}} (A^{k-1} \circ \D_{k-1} (A))^2 (x_1,\dots,x_{k-1}) \,.
\end{equation}

\bp

A finite set $\Lambda \subseteq \Gr$ is called {\it dissociated} if any  equality of the form 
\[
    \sum_{\la \in \Lambda} \eps_\la \la = 0 \,,
    \quad \quad 
    \mbox{ where }
    \quad \quad 
    \eps_\la \in \{0,\pm 1\} \,, \quad \quad  \forall \la \in \Lambda 
\]
implies $\eps_\la = 0$ for all $\la \in \Lambda$. 
Let $\dim (A)$ be the size of the largest dissociated subset of $A$ and we call $\dim(A)$ the {\it additive dimension} of $A$. 
Let us denote all combinations of the form $\{ \sum_{\la \in \Lambda} \eps_\la \la \}_{\eps_\la \in \{0,\pm 1\}}$ as $\Span (\Lambda)$.

\bp



We need 
the generalized triangle inequality \cite[Theorem 7]{SS_higher}.

\begin{lemma}
    Let $k_1,k_2$ be positive integers, $W\subseteq \Gr^{k_1}$, $Y\subseteq \Gr^{k_2}$ and $X,Z \subseteq \Gr$. 
    Then
\begin{equation}\label{f:gen_triangle_S}
    |W\times X| |Y-\Delta_{k_2} (Z)| \le |W\times Y \times Z - \Delta_{k_1+k_2+1} (X)| \,. 
\end{equation}
\label{l:gen_triangle_S}
\end{lemma}
The signs $\ll$ and $\gg$ are the usual Vinogradov symbols. 
All logarithms are to base $e$.

\section{The model case}
\label{sec:proof}

In this section we follow the well--known logic (see \cite{Green_models}) and first  consider the simplest case $\Gr=\F_2^n$, where the technical difficulties are minimal, and the case of general groups will be considered later. 

We need a generalization of 
a result from \cite[Remark 6]{sh_uncertainty} (also see the proof of \cite[Theorem 4]{sh_uncertainty}).
For the convenience  of the reader, we recall our argument. 

\begin{lemma}
    Let $A,B\subseteq \Gr$ be sets, $|A-A| = K|A|$ and  $k\ge 2$ be an integer.
    Then 
\begin{equation}\label{f:E(A,D)}
    \E_k (B) \E^k (A,A+B) \ge \frac{|A|^{2k+1} |B|^{2k}}{K} \,.
\end{equation}
\label{l:E(A,D)}
\end{lemma}
\begin{proof}
Let $S=A+B$.
In view of inclusion \eqref{f:Katz-Koester}, we have
\[
    \E (A,S) = \sum_{x} |A_x| |S_x| \ge \sum_{x} |A_x| |B+A_x| \,.
\]
    Using Lemma \ref{l:gen_triangle_S}, we see that for any $Z\subseteq \Gr^l$ the following holds 
\[
    |Z| \mathcal{D}_k := |Z| |B^{k-1} - \D_{k-1} (B)| \le |B-Z|^k
\]
    and therefore by the H\"older inequality one has 
\begin{equation}\label{tmp:12.5_1}
    \E (A,S) \ge \mathcal{D}^{1/k}_k \sum_x |A_x|^{1+1/k} 
    \ge 
    \mathcal{D}^{1/k}_k  \left( \sum_x |A_x| \right)^{1+1/k} |A-A|^{-1/k} \,.
\end{equation}
    Thus 
\begin{equation}\label{tmp:D_k_1}
    \E^k (A,S) \ge \mathcal{D}^{}_k |A|^{2k+1} K^{-1} \,. 
\end{equation}
    On the other hand, applying the Cauchy--Schwarz inequality and formula \eqref{f:Ek_Cf}, we derive
\begin{equation}\label{tmp:D_k_2}
    |B|^{2k} = \left( \sum_{y\in \Gr^{k-1}} (B^{k-1} \circ \D_{k-1}(B) (y)  \right)^2 
    \le 
    \mathcal{D}^{}_k \sum_{y\in \Gr^{k-1}} (B^{k-1} \circ \D_{k-1}(B)^2 (y) 
    =
    \mathcal{D}^{}_k \E_{k} (B) \,.    
\end{equation}
    Combining \eqref{tmp:D_k_1} and \eqref{tmp:D_k_2},
    one obtains 
\begin{equation}\label{tmp:D_k_2'}
    \E_k (B) \E^k (A,S) \ge \E_k (B) \mathcal{D}^{}_k |A|^{2k+1} K^{-1} 
    \ge 
    |B|^{2k} |A|^{2k+1} K^{-1} 
\end{equation}
    as required. 
$\hfill\Box$
\end{proof}

\bp 

Now we are ready to prove our main technical proposition.  

\begin{proposition}
    Let $\Gr = \F_2^n$, $A,B \subseteq \Gr$ be sets, 
    $|A|=\d N$, 
    $|B|= \omega |A|$,  $|A+B| = K|A|$, 
    $|A-A|= K'|A|$,
    and $0 < M \le K$, $0<M' \le K'$, 
    $\kappa>0$, $\zeta \in (0,1)$, $1<T\le M' (M+\kappa) \omega^{-1}$
    be some parameters. 
    Suppose that 
\begin{equation}\label{cond:param_rho+E}
    \M^2 (A) \le \frac{M|A|^2}{K},\, 
    \quad \quad 
    \E(B) \le \frac{M'|B|^3}{K'}\,, 
\end{equation}
    and 
\begin{equation}\label{cond:param}
    |A+B|^2 \le \kappa |A| N \,.
\end{equation}
    Then there is $\mathcal{L} \le \Gr$ ($\mathcal{L}$ depends on $B$ only) and  $z\in \Gr$ such that 
\[
    |B\cap (\mathcal{L} + z)| \ge \frac{(1-\zeta)\omega |\mathcal{L}|}{T(M+\kappa)} \,, 
\]
    and 
\[
    \codim (\mathcal{L}) \ll (\omega \zeta)^{-2} T^2 (M+\kappa)^2
        \cdot \left( \log ( \d^{-1} K') +  \log_T (M'(M+\kappa)\omega^{-1}) \cdot \log ((M+\kappa)\o^{-1}) \right) \,.
\]
\label{p:param}
\end{proposition}
\begin{proof}
    Let $a=|A|$, $b= |B| = \beta N = \o \d N$, $S=A+B$ and 
    $\E_k = \E_k (B)$. 
    Thanks to the condition \eqref{cond:param} and Parseval identity \eqref{F_Par}, we have 
\[
    \E(A,S) \le \frac{a^2 |A+B|^2}{N} + Ma^3 \le (M+\kappa) a^3 
    \,.
\]
    Using Lemma \ref{l:E(A,D)}, we obtain for all integers  $k\ge 2$ that 
\begin{equation}\label{tmp:Ek_low}
    \E_{k+1} \ge 
    \frac{b^{2k+2}}{K' (M+\kappa)^{k+1} a^k} 
    =
    \frac{b^{k+2} \omega^k}{K' (M+\kappa)^{k+1}} \,.
\end{equation}
    Suppose that for any $k\ge 2$ the following holds 
\[
    \E_{k+1} \le \frac{b \E_k}{M_*} \,,
\]
    where $M_*\ge 1$ is a parameter. 
    Then thanks to \eqref{cond:param_rho+E}, we derive 
\begin{equation}\label{tmp:Ek+1_M'}
    \E_{k+1} \le \frac{b^{k-1} \E_2}{M^{k-1}_*} \le \frac{M' b^{k+2}}{K' M^{k-1}_*} \,,
\end{equation}
    and using \eqref{tmp:Ek_low}, we have 
\[
    M' (M+\kappa)^{k+1} \ge \omega^k M^{k-1}_* \,.
\]
    Put $M_* = \omega^{-1} (M+\kappa) T$.
    It gives us 
\[
    M' (M+\kappa)^2 \ge \omega T^{k-1} 
\]
    and we obtain a contradiction if $k \ge k_0 := \lceil 10 \log_T (M'(M+\kappa)\omega^{-1}) \rceil + 10$, say. 
    Thus there is $2\le k \le k_0$ such that
\begin{equation}\label{tmp:21.11_1}
    \E_{k+1} \ge \frac{b \E_k}{M_*} \,.
\end{equation}
    Consider the function  $\_phi(x) = |B_x|^k$.
    Clearly, $\FF{\_phi} \ge 0$ (use, for example, formula  \eqref{f:F_svertka}) and $\| \_phi\|_1 = \E_k$. 
    In terms of Fourier transform we can rewrite inequality \eqref{tmp:21.11_1} as 
\[
    \E_{k+1} = \frac{1}{N} \sum_\xi |\FF{B} (\xi)|^2 \FF{\_phi} (\xi) \ge \frac{b \E_k}{M_*} = \frac{\omega b \E_k}{T(M+\kappa)} \,.
\]
    Now we use the parameter $\zeta$ and derive
\begin{equation}\label{f:correlation_Spec}
    \frac{1}{N} \sum_{\xi \in \Spec_{\zeta/M_*} (\_phi)} |\FF{B} (\xi)|^2 \FF{\_phi} (\xi) 
        \ge  
    \frac{(1-\zeta) \omega b \E_k}{T(M+\kappa)} \,.
\end{equation}
    Let $\Lambda \subseteq \Spec_{\zeta/M_*} (\_phi)$ be a dissociated set such that $|\Lambda| = \dim (\Spec_{\zeta/M_*} (\_phi))$. 
    Put $\mathcal{L}^\perp = \Span \Lambda$. 
    Then 
\begin{equation}\label{tmp:sigma_L(A)}
    \frac{(1-\zeta) \omega b \E_k}{T(M+\kappa)} 
    \le 
        \frac{1}{N} \sum_{\xi \in \mathcal{L}^\perp} |\FF{B} (\xi)|^2 \FF{\_phi} (\xi)
    \le 
        \frac{\E_k}{N} \sum_{\xi \in \mathcal{L}^\perp} |\FF{B} (\xi)|^2 
    =
    \frac{\E_k}{|\mathcal{L}|} \sum_x (B\circ B) (x) \mathcal{L} (x) \,.
\end{equation}
    By the pigeonhole principle there is $z\in B$ such that 
\begin{equation}\label{f:A_L+z}
    |B\cap (\mathcal{L} + z)| \ge \frac{(1-\zeta)\omega|\mathcal{L}|}{T(M+\kappa)} \,.
\end{equation}
    Using the Chang Lemma \ref{l:Chang} and estimate \eqref{tmp:Ek_low}, we derive
\begin{equation}\label{tmp:12.09_1}
    \codim (\mathcal{L}) 
        \ll 
    \zeta^{-2} M^2_* \log \left( \frac{\| \_phi\|_2^2 N}{\| \_phi \|^2_1} \right)
    \ll
     (\omega \zeta)^{-2} T^2 (M+\kappa)^2 \log \left( \frac{b^k N}{\E_k} \right)
\end{equation}
\begin{equation}\label{tmp:12.09_2}
         \ll 
    (\omega \zeta)^{-2} T^2 (M+\kappa)^2 \log\left( \frac{K' (M+\kappa)^k}{\beta \o^{k-1}}  \right) \,.
\end{equation}
    Recalling that $k\le k_0$, we finally obtain 
\[
    \codim (\mathcal{L}) 
        \ll 
    (\omega \zeta)^{-2} T^2 (M+\kappa)^2
        \cdot 
\]
\begin{equation}\label{tmp:12.09_3}
    \left( \log ( \d^{-1} K')
        +  \log_T (M'(M+\kappa) \omega^{-1}) \cdot \log ((M+\kappa) \o^{-1}) \right) \,.
\end{equation}
    This completes the proof. 
$\hfill\Box$
\end{proof}

\begin{remark}
    Suppose that $A=B$. 
    Then in terms of Proposition \ref{p:param} one has 
\[
    \E(A) \le \frac{|A|^4}{N} + \M^2 (A) |A| 
    \le 
    \frac{|A|^4}{N} + \frac{M|A|^3}{K}
    \le 
    \frac{M|A|^3}{K} \left( 1+ \frac{\d K}{M} \right)
    \le 
    \frac{M|A|^3}{K} \left( 1+ \frac{\kappa}{KM} \right) \,.
\]
    Thus $M' \le \left( 1+ \frac{\kappa}{KM} \right) M$. 
\label{r:M_M'}
\end{remark}



We derive some consequences of Proposition \ref{p:param}. 
Let us start with the case when $\M(A)$ is really small, namely, $\M^2(A) \le (2-\eps)|A|^2/K$. 

\begin{corollary}
    Let $\Gr = \F_2^n$, $A \subseteq \Gr$ be a set, 
    $|A|= \d N$, 
    $|A-A|=K|A|$, and $\eps \in (0,1)$ be a parameter. 
    Suppose that 
\[
    100 K^2 \d \le \eps \,.
\]
    Then either there is $x\neq 0$ such that
\[
    |\FF{A} (x)|^2 \ge \frac{(2-\eps)|A|^2}{K} \,,
\]
    or 
    there exists $\mathcal{L} \le \Gr$ with $\mathcal{L} \subseteq A-A$ and 
\[
    \codim (\mathcal{L}) \ll \eps^{-2} \log (\d^{-1} K )  + \eps^{-3} \,.
\]
\label{cor:2-eps}
\end{corollary}
\begin{proof}
    We apply Proposition \ref{p:param} with 
    $B=-A$, $\omega =1$, 
    $M=2-\eps$, $M' = \left( 1+ \frac{\kappa}{KM} \right) M \le M+\kappa$ (see Remark \ref{r:M_M'}), $T=1+\kappa$, and $\kappa = \zeta =\eps/100$, say. 
    Thus we find $\mathcal{L} \le \Gr$ and  $z\in \Gr$ such that 
\begin{equation}\label{tmp:24.11}
    |A\cap (\mathcal{L} + z)| \ge \frac{(1-\zeta)|\mathcal{L}|}{T(M+\kappa)}
    \ge  
    |\mathcal{L}| \left( \frac{1}{2} + \frac{\eps}{8} \right) \,,
\end{equation}
    and 
\[
\codim (\mathcal{L}) \ll \zeta^{-2} T^2 (M+\kappa)^2
        \cdot \left( \log (\d^{-1} K) +  \log_T (M'(M+\kappa)) \cdot \log (M+\kappa) \right) 
\]
\[
    \ll 
    \eps^{-2} \log (\d^{-1} K ) + \eps^{-3} \,.
\]
    The inequality \eqref{tmp:24.11} implies that $\mathcal{L} \subseteq A-A$. 
    This completes the proof. 
$\hfill\Box$
\end{proof}

\bp

Now we are ready to obtain our structural result for sets with small doubling and small $\M(A)$. 
Given two sets $A,B \subseteq \Gr$ we write $A\dotplus B$ if $|A+B| = |A||B|$.

\begin{corollary}
    Let $\Gr = \F_2^n$, $A \subseteq \Gr$ be a set, 
    $|A|= \d N$, 
    $|A-A|=K|A|$, and $1\le M \le K$ be a parameter. 
    Suppose that 
\[
    100 K^2 |A| \le N \,.
\]
    Then either there is $x\neq 0$ such that
\begin{equation}\label{f:M_large}
    |\FF{A} (x)|^2 > \frac{M|A|^2}{K} \,,
\end{equation}
    or for any $B\subseteq A$, $|B| =\beta N$ there exist $H \le \Gr$, $z\in \Gr$ with $H \subseteq 3B+z$ and 
\[
    \codim (H) \ll (\d \beta^{-1} M)^2 \left( \log ( \d^{-1} K) + \log^2 (\d \beta^{-1} M)\right ) \,.
\]
    In the last case one can find $\Lambda \subseteq \Gr/H$ such that 
\[
    |A\cap (\Lambda \dotplus H)| \ge \frac{|A|}{16M}
    \quad \quad 
    \mbox{ and }
    \quad \quad 
    |\Lambda| |H| \le 16M |A| \,.
\]
\label{cor:M}
\end{corollary}
\begin{proof}
    If $\M^2 (A) > M|A|^2/K$, then there is nothing to prove. 
    Otherwise, $\M^2 (A) \le M|A|^2/K$ and since 
    $|A-B|\le |A-A| = K|A|$, it follows that  condition \eqref{cond:param} of  Proposition \ref{p:param} takes place. 
    We apply 
    this proposition 
    with 
    $\kappa =1$, $T=2$, $\zeta = 1/8$, $|B| = \beta N := \o |A|$ and $M'=2M$ (see Remark \ref{r:M_M'}) to find $\mathcal{L} \le \Gr$ such that 
\[
    \codim (\mathcal{L}) 
        \ll 
        (\omega \zeta)^{-2} T^2 (M+\kappa)^2
        \cdot \left( \log (\d^{-1} K) +  \log_T (M'(M+\kappa) \o^{-1}) \cdot \log ((M+\kappa)\o^{-1}) \right)  
\]
\[
        \ll 
        (\d \beta^{-1} M)^2 \left( \log (\d^{-1} K) +  \log^2 (\d \beta^{-1} M) \right) 
        \,,
\]
    and (see  \eqref{tmp:sigma_L(A)} or \eqref{tmp:24.11})
\begin{equation}\label{f:density_8M}
    |B\cap (\mathcal{L} + z)| \ge 
    \frac{(1-\zeta)|\mathcal{L}|}{T(M+\kappa)}
    \ge 
    \frac{|\mathcal{L}|}{8M} \,.
\end{equation}
    Put $B'= B\cap (\mathcal{L} + z)$.
    It remains to 
    apply the Kelley--Meka bound (see \cite{kelley2023strong} and \cite[Theorem 3]{bloom2023kelley}) and find $H\le \mathcal{L}$, $z\in \Gr$ such that $H \subseteq  3B'+z$  and 
\[
    \codim (H) \le \codim(\mathcal{L}) + O(\log^{O(1)} M) \ll 
    (\d \beta^{-1} M)^2 \left( \log (\d^{-1} K) + \log^2  (\d \beta^{-1} M) \right) \,.
\]
    Returning to \eqref{tmp:sigma_L(A)} with $B=A$, we see that 
\[
    \sum_x (A\circ A) (x) \mathcal{L} (x) \ge \frac{|\mathcal{L}||A|}{8M} \,. 
\]
    Put $A = \bigsqcup_{\la\in \Gr/\mathcal{L}} (A\cap (\mathcal{L} + \la))$. 
    Then the last inequality is equivalent to 
\[
    2\sum_{\la\in \Gr/\mathcal{L} ~:~ |A_\la| \ge |\mathcal{L}|/16M} |A_\la|^2 \ge \sum_{\la\in \Gr/\mathcal{L}} |A_\la|^2 \ge \frac{|\mathcal{L}||A|}{8M} \,.
\]
    Let $\Lambda = \{ \la\in \Gr/\mathcal{L} ~:~ |A_\la| \ge |\mathcal{L}|/16M\}$. 
    Putting $A_* = A\cap (\Lambda \dotplus \mathcal{L})$, we see that $|A_*| \ge |A|/16M$ and $|\Lambda| |\mathcal{L}| \le 16M|A|$.
    This completes the proof.  
$\hfill\Box$
\end{proof}

\bp

The following  consequence of Proposition \ref{p:param} allows us to have a ``regularization'' of any set $A$ in terms of the doubling constant of  $A$ and its density, see inequality \eqref{f:cor_density} below.
Given a set $A\subseteq \Gr$ put $K[A]:= |A-A|/|A|$.

\begin{corollary}
    Let $\Gr = \F_2^n$, $A\subseteq \Gr$ be a set, $|A|=\d N$.
    Then there is $H\le \Gr$ and $z\in \Gr$ such that 
\[
    \codim (H) \ll \d^{-2} \log^{3} (1/\d) \,.
\]
    and for $\tilde{A} = A \cap (H+z)$ one has $\tilde{\d} = |\tilde{A}|/|H| \ge \d$, $\tilde{K} = K[\tilde{A}]$ and 
\begin{equation}\label{f:cor_density}
    100 \tilde{K}^2 \tilde{\d} > 1 \,.
\end{equation}
\label{cor:density}
\end{corollary}
\begin{proof}
    Let $K=K[A]$. If $100 K^2 \d >1$, then there is nothing to prove. 
    Otherwise,  we apply Proposition \ref{p:param} with 
    $B=-A$, $\omega = 1$, $\kappa =1$, $T=2$, $\zeta = 1/8$ and $M'=2M$ as we did in the proof of Corollary \ref{cor:M}.
    Here $M$ is defined as $\M(A)^2 = M|A|^2/K$, where $K=K[A]$.  
    Exactly as in \eqref{f:density_8M} we  find $H_1 \le \Gr$ and $z_1\in \Gr$ such that 
\begin{equation}\label{tmp:H1}
    |A\cap (H_1 + z_1)| \ge  \frac{|H_1|}{8M} \,, 
\end{equation}
    and 
\[
    \codim (H_1) 
        \ll 
        M^2 \left( \log (\d^{-1} K) + \log^2 M \right) 
        \,.
\]
    Now suppose that $M\le 1/(16 \d)$ and therefore $\codim (H_1) 
        \ll \d^{-2} \log^2 (\d^{-1})$. 
    Then \eqref{tmp:H1} shows that the $A$ has density $2\d$ inside $H_1+z_1$. 
    On the other hand, if $M > 1/(16 \d)$ then there is $x\neq 0$ such that 
\[
    |\FF{A} (x)| \ge |A| \cdot (M/K)^{1/2} = 2^{-2} \d |A| \,.
\]
    Here we have used the trivial  fact that $K\le \d^{-1}$.
    In this case we can use classical density increment (see \cite{Roth1953} or \cite{TV}) and find $H'_1 \le \Gr$, $\codim (H'_1) = 1$, $z'_1 \in \Gr$ such that 
\begin{equation}\label{tmp:H1'}
    |A\cap (H'_1 + z'_1)| \ge  
    \d ( 1+  2^{-3} \d) |H'_1| \,.
\end{equation} 
    After that, one can repeat 
    our dichotomy 
    for $A_1:=A\cap (H_1 + z_1)$ or $A'_1:=A\cap (H'_1 + z'_1)$. 
    Using \eqref{tmp:H1} or \eqref{tmp:H1'}, we see that the algorithm  must stop after a finite number of steps and the resulting codimension is big--O of 
\[
    \log (1/\d) \cdot \d^{-2} \log^2 (\d^{-1})
     + \d^{-1} 
    =
    O(\d^{-2} \log^{3} (1/\d)) \,.
\]
   This completes the proof.  
$\hfill\Box$
\end{proof}

\section{General case}
\label{sec:general_case}

Now we are ready to consider the case of an arbitrary finite abelian group $\Gr$. The main tool here is the so--called {\it Bohr sets}, and we recall all the necessary  definitions and properties of this object. 
Bohr sets were introduced to additive number theory by Ruzsa \cite{ruzsa_GAP_sumsets} and Bourgain \cite{Bourgain_AP3_first} was the first who used Fourier analysis on Bohr sets to improve the estimate in Roth's theorem \cite{Roth1953}. 
Sanders (see, e.g., \cite{sanders_certain_0.75}, \cite{sanders2012bogolyubov}) developed the theory of Bohr sets proving many important theorems, see for
example Lemma \ref{l:Chang_Bohr}  below.

\begin{definition}
Let $\G$ be a subset of $\FF{\Gr}$, $|\G| = d$, and $\eps = (\eps_1,\dots,\eps_d) \in (0,1]^d$.
Define the Bohr set  $\B = \B(\G, \eps)$ by
\[
  \B(\G, \eps) = \{ n \in \Gr ~|~ \| \g_j \cdot n\| < \eps_{j} \mbox{ for all } \g_j \in \G \} \,,
 \] where $\|x\|=|\arg x|/2\pi.$
\end{definition}

The number  $d=|\G|$ is called  {\it dimension }
of  $\B$ and is denoted by $\dim B$. If $M = \B + n$, $n\in
\Gr$ is a translation of  $\B$, then, by definition, put $\dim (M) =
\dim (\B)$.
The {\it intersection} $\B\wedge \B'$ of two Bohr sets $\B = \B(\G,\eps)$ and $\B' = \B(\G',\eps')$ is the Bohr set
with the generating set $\G\cup \G'$ and new vector $\t{\eps}$ equals $\min \{ \eps_j,\eps'_j \}$.
Furthermore, if $\B=\B(\G,\eps)$ and $\rho>0$ then by $\B_\rho$ we mean $\B(\G,\rho \eps).$


\begin{definition}
A Bohr set $\B = \B (\G, \eps)$ is called {\it regular},
if for every  $\eta,\, d|\eta|\ls  1/100$
we have
\begin{equation}\label{cond:reg_size}
  (1-100d|\eta|)|\B_1| < { |\B_{1+\eta} | } < (1+ 100d|\eta|)|\B_1| \,.
\end{equation}
\end{definition}


We  formulate  a sequence of basic properties of Bohr, which will be used later.


\begin{lemma}
Let  $\B (\G,\eps)$ be a Bohr set. Then there exists $\eps_1$ such that $\frac{\eps}{2} < \eps_1 < \eps$ and   $\B (\G,\eps_1)$ is regular.
\label{l:Reg_B}
\end{lemma}


\begin{lemma}
Let $\B (\G,\eps)$ be a Bohr set. Then
\[
  |\B (\G,\eps)| \ge \frac{N}{2} \prod_{j=1}^d \eps_j \,.
\]
\label{l:Bohr_est}
\end{lemma}


\begin{lemma}
    Let $\B (\G,\eps)$ be a Bohr set.
    Then
    $$
        |\B(\G,\eps)| \ls 8^{|\G|+1} |\B(\G,\eps/2)| \,.
    $$
\label{l:entropy_Bohr}
\end{lemma}

\begin{lemma}
    Suppose that $\B^{(1)}, \dots, \B^{(k)}$ is a sequence of Bohr sets.
    Then
    $$
        |\bigwedge_{i=1}^k \B^{(i)}| \ge N\cdot \prod_{i=1}^k \frac{|\B^{(i)}_{1/2}|}{N} \,.
    $$
\label{l:Bohr_intersection_Sanders}
\end{lemma}

Recall a local version of Chang's lemma \cite{chang2002polynomial}, see \cite[Lemma 5.3]{sanders2012bogolyubov} and \cite[Lemmas 4.6, 6.3]{sanders_certain_0.75}.

\begin{lemma}
    Let $\eps,\nu,\rho \in (0,1]$ be positive real numbers. 
    Suppose that $\B$ is a regular Bohr set and $f:\B \to \C$.
    Then there is a set $\Lambda$ of size $O(\eps^{-2} \log (\|f\|_2^2 |\B|/\|f\|^2_1))$ such that for any $\gamma \in \Spec_\eps (f)$ we have
\[
    |1-\gamma (x)| \ll |\Lambda| (\nu + \rho \dim^2 (B))
    \quad \quad \forall x\in \B_\rho \wedge \B'_\nu \,,
\]
    where $\B'=\B(\Lambda,1/2)$. 
\label{l:Chang_Bohr}
\end{lemma}

Now we are ready to obtain an analogue of Proposition \ref{p:param}.

\begin{proposition}
    Let $\Gr$ be a finite abelian group, 
    $A,B \subseteq \Gr$ be sets, 
    $|B| = \omega |A|$,  $|A+B| = K|A|$, 
    $|A-A|= K'|A|$,
    and $0 < M \le K$, $0<M' \le K'$,  
    $\kappa>0$, $\zeta \in (0,1/2)$,  $1<T\le M' (M+\kappa) \omega^{-1}$ 
    be some parameters. 
    Suppose that 
\begin{equation}\label{cond:param_rho+E_Bohr}
    \M^2 (A) \le \frac{M|A|^2}{K},\, 
    \quad \quad 
    \E(B) \le \frac{M'|A|^3}{K'}\,, 
\end{equation}
    and 
\begin{equation}\label{cond:param_Bohr}
    |A+B|^2 \le \kappa |A| N\,.
\end{equation}
    Then there is a regular Bohr set $\B_* = \B(\G,\eps)$ 
    ($\B_*$ depends on $B$ only) and  $z\in \Gr$ such that 
\[
    |B\cap (\B_* + z)| \ge \frac{(1-2\zeta)|\B_*|}{T(M+\kappa)} \,, 
\]
    and 
\begin{equation}\label{p:param_Bohr_B*_1}
    \dim (\B_*) \ll (\omega \zeta)^{-2} T^2 (M+\kappa)^2 
        \cdot \left( \log ( \d^{-1} K) +  \log_T (M'(M+\kappa) \o^{-1}) \cdot \log ((M+\kappa) \o^{-1}) \right)  
    \,,
\end{equation}
    as well as 
\begin{equation}\label{p:param_Bohr_B*_2}
     |\B_*| 
    \gg 
    N \cdot \exp (-O(\dim (\B_*) \log ( (\omega \zeta)^{-1} (M+\kappa) T \dim (\B_*)))) \,. 
\end{equation}
\label{p:param_Bohr}
\end{proposition}
\begin{proof}
    We use the same argument and the notation 
    of the proof of 
    Proposition \ref{p:param}. 
    The argument before inequality \eqref{f:correlation_Spec} does not depend on a group, therefore we have 
\begin{equation}\label{f:correlation_Spec_Bohr}
    \frac{1}{N} \sum_{\xi \in \Spec_{\zeta/M_*} (\_phi)} |\FF{B} (\xi)|^2 \FF{\_phi} (\xi) 
        \ge  
    \frac{(1-\zeta) \omega b \E_k}{T(M+\kappa)} \,.
\end{equation}
    Applying Lemma \ref{l:Chang_Bohr} (combining with the triangle inequality) with $\B=\Gr$ (hence $\dim (\B)=1$), $f=\_phi$, and $\rho = \nu = c \zeta/(M_*|\Lambda|)$, where $c>0$ is a sufficiently small absolute constant, we see that for any $\xi \in \Spec_{\zeta/M_*} (\_phi)$ one has 
\[
    \Spec_{\zeta/M_*} (\_phi) (\xi) \le |\B_*|^{-2} |\FF{\B_*} (\xi)|^2 (1+\zeta) \,, 
\]
    where $\B_* = \B_\rho \wedge \B'_\nu$. 
    Thus \eqref{f:correlation_Spec_Bohr} gives us 
\[
    \frac{1}{N} \sum_{\xi} |\FF{B} (\xi)|^2 |\FF{\B_*} (\xi)|^2
        \ge  
    \frac{(1-2\zeta) \omega b |\B_*|^2}{T(M+\kappa)} \,.
\]
    This  is equivalent to 
\[
    \sum_{x} (B\circ B)(x) (\B_* \circ \B_*) (x) \ge \frac{(1-2\zeta) \omega b |\B_*|^2}{T(M+\kappa)} \,.
\]
    By the pigeonhole principle there is $z\in \Gr$ such that 
\begin{equation*}
    |B\cap (\B_* + z)| \ge \frac{(1-2\zeta) \omega |\B_*|}{T(M+\kappa)} \,.
\end{equation*}
    Now 
    \[
        \dim (\B_*) = |\Lambda| \ll (\omega \zeta)^{-2} T^2 (M+\kappa)^2 
        \cdot \left( \log (\d^{-1} K) +  \log_T (M'(M+\kappa) \o^{-1}) \cdot \log ((M+\kappa) \o^{-1}) \right) 
        \,,
    \] 
    see computations in \eqref{tmp:12.09_1}---\eqref{tmp:12.09_3}. 
    Applying Lemmas \ref{l:Bohr_est}, \ref{l:Bohr_intersection_Sanders}, we see that 
\[
    |\B_*| \gg N \cdot (\zeta/(M_*|\Lambda|))^{O(|\Lambda|)} 
    \gg 
    N \cdot \exp (-O(\dim (\B_*) \log ( (\omega \zeta)^{-1} (M+\kappa) T \dim (\B_*)))) \,. 
\]
    Finally, in view of Lemma \ref{l:Reg_B} one can assume that $\B_*$ is a regular Bohr set. 
   This completes the proof.  
$\hfill\Box$
\end{proof}

\bp 

Proposition \ref{p:param_Bohr} immediately implies an analogue of Corollary \ref{cor:M}. 

\begin{corollary}
    Let $\Gr$ be a finite abelian group, $A \subseteq \Gr$ be a set, 
    $|A|= \d N$, 
    $|A-A|=K|A|$, and $1\le M \le K$ be a parameter. 
    Suppose that 
\begin{equation}\label{cond:M_large_Bohr}
    100 K^2 |A| \le N \,.
\end{equation}
    Then either there is $x\neq 0$ such that
\begin{equation}\label{f:M_large_Bohr}
    |\FF{A} (x)|^2 > \frac{M|A|^2}{K} \,,
\end{equation}
    or for any $B\subseteq A$ or $B \subseteq -A$, $|B| = \beta N$ there exists  a regular Bohr set $\B_* = \B(\G,\eps)$  and  $z\in \Gr$ such that 
\[
    |B\cap (\B_* + z)| \ge \frac{|\B_*|}{8M} \,, 
\]
    and the bounds 
\begin{equation}\label{cor:M_Bohr_dim}
    \dim (\B_*) \ll (\d \beta^{-1} M)^2 \left( \log (\d^{-1} K) + \log^2 (\d \beta^{-1} M) \right) \,,
\end{equation}
\begin{equation}\label{cor:M_Bohr_size}
 |\B_*| 
    \gg 
    N \cdot \exp (- O(\dim (\B_*) \log (\d \beta^{-1} M \dim (\B_*)))) 
\end{equation}
    take place. 
\label{cor:M_Bohr}
\end{corollary}

Now we present a construction showing that the estimates in  Corollary \ref{cor:M_Bohr} are close to exact.


\begin{example}
    Let $p$ be a prime number, $d=4$ and $\Gr $ be the cyclic group 
    $\Gr = \F^*_{p^d}$. Put $A = \{ \mathrm{ind} (g+j)\} _{j=0,1,\dots,p-1}$, where $g$ is a fixed element that generates $\Gr$.
    Then by Katz's result (see \cite[Theorem 1]{Katz_char_est} and the proof of \cite[Lemma 1]{Andersson_Katz}) all non--zero Fourier coefficients of $A$ are bounded by $(d-1)\sqrt{p} = 3 \sqrt{|A|}$. 
    Clearly, $K=|A-A|/|A| \le |A|$ and hence 
\begin{equation}\label{tmp:M_exm}
    \M^2 (A) \le (d-1)^2 |A|\le  \frac{(d-1)^2 |A|^2}{K} \ll \frac{|A|^2}{K} \,.
\end{equation}
    In other words, the set $A$ has small Fourier coefficients. 
    Also, 
\[
    K^2 \d \le |A|^3 N^{-1} = o(1) 
\]
    and, therefore,  condition \eqref{cond:M_large_Bohr} is satisfied  for large $N$. 
    Now, if there exists a regular Bohr set $\B=\B(\Gamma,\eps)$ and $z\in \Gr$ such that  
    $|A\cap (\B+z)| \gg |\B|$, then $|\B| \ll |A| \ll N^{1/4}$
    (if one believes in GRH \cite{IK_book}, then  it is possible to obtain even better upper bounds for the cardinality of the intersection $A\cap (\B+z)$, see the argument of  \cite{Hanson_Bohr}). 
    But then estimate  \eqref{cor:M_Bohr_size} gives us 
    $$
    \dim (\B) \gg \frac{\log N}{\log \log N} \ge \frac{\log (\d^{-1}K)}{\log \log (\d^{-1}K)} \,,
    $$
    and this coincides with \eqref{cor:M_Bohr_dim} up to double logarithm.
\label{exm:Andersson}
\end{example}

Similarly, it is easy to prove an analogue of Corollary \ref{cor:2-eps} (and we leave the derivation of  the analogue of Corollary \ref{cor:density} to the interested reader).

\begin{corollary}
    Let $\Gr$ be a finite abelian group, $A \subseteq \Gr$ be a set, $|A|= \d N$, $|A-A|=K|A|$, and $\eps \in (0,1)$ be a parameter. 
    Suppose that 
\[
    100 K^2 \d \le \eps \,.
\]
    Then either there is $x\neq 0$ such that
\[
    |\FF{A} (x)|^2 \ge \frac{(2-\eps)|A|^2}{K} \,,
\]
    or there is a regular Bohr set $\B_* = \B(\G,\eps)$  and  $z\in \Gr$ such that 
    $\B_*\subseteq A-A$ and 
\begin{equation}\label{p:param_Bohr_B*_1'}
    \dim (\B_*) \ll \eps^{-2} \log (\d^{-1} K ) + \eps^{-3} \,,
\end{equation}
    as well as 
\begin{equation}\label{p:param_Bohr_B*_2'}
    |\B_*| 
    \gg 
    N \cdot 
    \exp (-O(\dim (\B_*) \cdot \log (\eps^{-1} \dim (\B_*) ))) \,. 
\end{equation}
\label{cor:2-eps_Bohr}
\end{corollary}
\begin{proof}
    We apply the argument of Proposition \ref{p:param_Bohr} with 
    $B=-A$, $\omega =1$, 
    $M=2-\eps$, $M' = \left( 1+ \frac{\kappa}{KM} \right) M \le M+\kappa$ (see Remark \ref{r:M_M'}), $T=1+\kappa$, and $\kappa = \zeta =\eps/200$, say. 
    Thus we find a Bohr set  $\B_* \subseteq \Gr$ and  $z\in \Gr$ such that 
\begin{equation}\label{tmp:24.11'}
     (A * \mu) (z) \ge \frac{(1-2\zeta)|\mathcal{L}|}{T(M+\kappa)}
    \ge  
    \left( \frac{1}{2} + \frac{\eps}{8} \right) \,,
\end{equation}
    where $\mu$ is any measure on $\B_*$.
    Also, we have 
\[
    d:=\dim (\B_*) \ll \zeta^{-2} T^2 (M+\kappa)^2
        \cdot \left( \log (\d^{-1} K) +  \log_T (M'(M+\kappa)) \cdot \log (M+\kappa) \right) 
\]
\[
    \ll 
    \eps^{-2} \log (\d^{-1} K ) + \eps^{-3} \,,
\]
    and thanks to Lemmas \ref{l:Bohr_est}, \ref{l:Bohr_intersection_Sanders}, one has  
\[
    |\B_*| \gg N \cdot (\zeta/(M_*|\Lambda|))^{O(|\Lambda|)} 
    \gg 
    N \cdot \exp (-O(\dim (\B_*) \cdot \log (\eps^{-1} \dim (\B_*) ))) \,.
\]
    Now we follow the argument of \cite[Lemma 9.2]{sanders2012bogolyubov}.
    Namely, put $t=\lceil 100 \eps^{-1} d \rceil$ and $\eta=1/2t$ and consider the sequence of Bohr sets  
\[
    (\B_*)_{1/2} \subseteq (\B_*)_{1/2+\eta} \subseteq  \dots \subseteq (\B_*)_{1/2+t\eta} = \B_* \,.
\]
    Applying Lemma \ref{l:entropy_Bohr}, we see that there is $j\in [t]$ such that 
\begin{equation}\label{tmp:B',B''}
    |\B''|:=|(\B_*)_{1/2+j\eta}| \le 8^{(d+1)/t} |(\B_*)_{1/2+(j-1)\eta}| \le (1+\eps/4) |(\B_*)_{1/2+(j-1)\eta}| := (1+\eps/4) |\B'|\,.
\end{equation}
    Consider the measure  $\mu (x) = \frac{\B'(x) + \B''(x)}{|\B'|+|\B''|}$, $\supp (\mu) \subseteq \B_*$. 
    Then inequality \eqref{tmp:24.11'} gives us 
\begin{equation}\label{tmp:B'+B''}
    |A\cap (\B'+z)| + |A\cap (\B''+z)| \ge \left( \frac{1}{2} + \frac{\eps}{8} \right) \cdot (|\B'| + |\B''|) \,.
\end{equation}
    One the other hand, for any $x\in (\B_*)_\eta$, we have 
\[
    (A \circ A) (x) \ge ( (A\cap (\B'+z)) \circ (A\cap (\B''+z)) ) (x)
\]
\[
    \ge 
    |A\cap (\B'+z)| + |A\cap (\B''+z)| - |(A\cap (\B'+z + x)) \bigcup (A\cap (\B''+z))|
\]
\[
    \ge 
    |A\cap (\B'+z)| + |A\cap (\B''+z)| - |A\cap (\B''+z)|
    \ge 
    |A\cap (\B'+z)| + |A\cap (\B''+z)| - |\B''| \,.
\]
    Using formulae \eqref{tmp:B',B''}, \eqref{tmp:B'+B''}, we get 
\[
    (A \circ A) (x) \ge \left( \frac{1}{2} + \frac{\eps}{8} \right) \cdot (|\B'| + |\B''|) - |\B''|
    = \left( \frac{1}{2} + \frac{\eps}{8} \right) |\B'| - \left( \frac{1}{2} -  \frac{\eps}{8} \right) \left( 1+ \frac{\eps}{4} \right) |\B'| 
\]
\begin{equation}\label{tmp:24.11''}
    \ge \frac{\eps |\B'|}{8} > 0 \,.
\end{equation}
    The inequality \eqref{tmp:24.11''} implies that $(\B_*)_\eta \subseteq A-A$. 
    We see that \eqref{p:param_Bohr_B*_1'}, \eqref{p:param_Bohr_B*_2'} take place 
    (use Lemmas \ref{l:Bohr_est}, \ref{l:Bohr_intersection_Sanders} again) 
    and thanks to Lemma  \ref{l:Reg_B} one can assume that we have deal with a  regular Bohr set.  
    This completes the proof. 
$\hfill\Box$
\end{proof}

\begin{problem}
    In Example \ref{exm:Andersson} we constructed a set $A \subseteq \Gr$, $|A| = \d |\Gr|$, $|A-A| = K|A|$ such that for $d>1$ one has 
\begin{equation}\label{f:pr_Kd}
    K^{d-1} \d \sim 1 \,,
\end{equation}
    and 
\begin{equation*}\label{f:pr_A}
    \M^2 (A) \le \frac{(d-1)^2 |A|^2}{K} \,,
\end{equation*}
    see Example \ref{tmp:M_exm}. 
    On the other hand, we always have a universal lower bound \eqref{f:kappa_0_intr}, provided $d \ge 4$. 
    Given $d>1$ and a set $A$ such that \eqref{f:pr_Kd} takes place, what are the proper upper/lower bounds for $\M (A)$? 
\end{problem}

\bibliographystyle{abbrv}

\bibliography{bibliography}{}
\noindent{I.D.~Shkredov\\
{\tt ilya.shkredov@gmail.com}

\end{document}


\section{Appendix}
\label{sec:appendix}

\begin{theorem}
    Let $\Gr$ be an abelian group, $A \subseteq \Gr$ be a set, $|A|= \d N$, $|D|=K|A|$, and $\eps \in (0,1)$ be a parameter. 
    Suppose that 
\[
    2 K^2 \d \le \eps \,.
\]
    Then  there is $x\neq 0$ such that
\[